\documentclass[a4paper]{article}
\usepackage{amsmath, amsfonts}
\usepackage{amssymb, latexsym}
\usepackage{amsthm}
\usepackage{mathtools}
\usepackage{comment}

\usepackage{tikz}
\usetikzlibrary{arrows.meta}
 \usetikzlibrary{patterns}

\usepackage{easy-todo}

\usepackage[shortlabels]{enumitem}
\setlist[enumerate]{leftmargin=*,align=left,labelindent=\parindent}

\usepackage[
colorlinks=true,       
linkcolor=blue,          
citecolor=blue,        
plainpages=false,      
pdfpagelabels,
]{hyperref}

\makeatletter
\def\widebreve{\mathpalette\wide@breve}
\def\wide@breve#1#2{\sbox\z@{$#1#2$}%
     \mathop{\vbox{\m@th\ialign{##\crcr
\kern0.08em\brevefill#1{0.8\wd\z@}\crcr\noalign{\nointerlineskip}%
                    $\hss#1#2\hss$\crcr}}}\limits}
\def\brevefill#1#2{$\m@th\sbox\tw@{$#1($}%
  \hss\resizebox{#2}{\wd\tw@}{\rotatebox[origin=c]{90}{\upshape(}}\hss$}
\makeatletter

\DeclareMathOperator{\imp}{\,\rightarrow\,}
\DeclareMathOperator{\defeqiv}{\stackrel{\textup{def}}{\iff}}
\DeclareMathOperator{\defeql}{\stackrel{\textup{def}}{\  =\  }}
\newcommand{\dotminus}{\mathbin{\ooalign{\hss\raise.6ex\hbox{$\cdot$}\hss\crcr$-$}}}

\newcommand{\ELZero}{\mathrm{EL_0}}
\newcommand{\EL}{\mathrm{EL}}

\newcommand{\RCA}{\mathrm{RCA_0}}

\newcommand{\BISH}{\mathrm{BISH}}
\newcommand{\qfAC}{\mathrm{QF\text{-}AC_{00}}}

\newcommand{\PiOneDJDC}{\mathrm{\Pi^{0}_{1}\text{-}DC^{\lor}}}
\newcommand{\SigmaOneAC}{\mathrm{\Sigma^{0}_{1}\text{-}AC_{00}}}

\newcommand{\FAN}{\mathrm{FAN}}

\newcommand{\WKL}{\mathrm{WKL}}
\newcommand{\WKLB}{\mathrm{WKL!}}

\newcommand{\LLPO}{\mathrm{LLPO}}

\newcommand{\MIN}{\mathrm{MIN}}
\newcommand{\MINB}{\mathrm{MIN!}}
\newcommand{\CIT}{\mathrm{CIT}}
\newcommand{\CITB}{\mathrm{CIT!}}
\newcommand{\POS}{\mathrm{POS}}

\newcommand{\Real}{\mathbb{R}}

\newcommand{\Nat}{\mathbb{N}}

\newcommand{\Cantor}{\Two^\Nat}
\newcommand{\Baire}{\Nat^\Nat}

\newcommand{\Bin}{\Two^\ast}

\newcommand{\Seq}{\Nat^\ast}
\newcommand{\Two}{\left\{0,1 \right\}}

\newcommand{\nil}{\langle \,\rangle}

\newcommand{\lth}[1]{\lvert #1 \rvert}
\newtheorem{theorem}{Theorem}[section]

\newtheorem{corollary}[theorem]{Corollary}

\theoremstyle{definition}

\theoremstyle{remark}

\numberwithin{equation}{section}

\title{The strength of the minimum principle in constructive reverse mathematics}

\author{Tatsuji Kawai\footnote{
  Email: \texttt{tatsuji.kawai@kochi-u.ac.jp}}\\
  \small
  Department of Information Science, Kochi University\\
  \small
  2-5-1 Akebono-cho, Kochi 780-8520 Japan
}
\date{}
\begin{document}
\maketitle
\begin{abstract}
The minimum principle states that every real-valued uniformly
continuous function $f \colon X \to \Real$ on a compact metric space
$X$ attains its minimum. Working in the base system $\ELZero$ of
constructive reverse mathematics, we show that the minimum principle
is equivalent to weak K\"{o}nig's lemma and to the Cantor intersection
theorem (i.e., the contrapositive of the Heine--Borel theorem).
Moreover, the above equivalence naturally extends
to the equivalence between the corresponding principles with unique
existence assumptions, which are then shown to be equivalent to the
Fan theorem for detachable bars and the positivity property for a
real-valued uniformly continuous function on a compact space.
These results sharpen the equivalence among these
principles obtained by Ishihara [Math.\ Log.\ Q. (1990), 237--240] 
and Berger and Ishihara [Math.\ Log.\ Q. (2005), 360--364]
in Bishop's constructive mathematics.
\medskip

 \\[3pt]
\end{abstract}


\section{Introduction}
\label{sec:Introduction}
The \emph{minimum principle} states:
\begin{description}
  \item[$\MIN$:]
    Every real-valued uniformly continuous function
    $f \colon X \to \Real$ on a compact metric space $X$ attains its
    minimum: i.e., there exists $z \in X$ such that $f(z) =
    \inf f$.
\end{description}
Here, as in \cite[Chapter 4, Section 4]{Bishop-67}, a metric space is
compact if it is complete and totally bounded.
Ishihara \cite{IshiharaLLPOWKLHahn} showed that $\MIN$ is equivalent
to weak K\"{o}nig's lemma ($\WKL$) in Bishop's constructive
mathematics ($\BISH$) \cite{Bishop-67,Bishop-Bridges-85,VarietiesConstMath}.
In \cite{IshiharaLLPOWKLHahn}, the equivalence 
was established through a cycle of implications that involve the
\emph{lesser limited principle of omniscience} ($\LLPO$) as an
intermediate step. Although $\BISH$ is based on intuitionistic logic,
it allows arbitrary use of countable choice. Thus, it is not suitable
for analysing function existence axioms inherent in the theorem being
classified, since such axioms may already be available through
the unrestricted use of countable choice.
In particular, over the base system $\RCA$ of classical
reverse mathematics, $\LLPO$ is derivable whereas $\WKL$ is not. Since
the equivalence of $\MIN$ and $\WKL$ in \cite{IshiharaLLPOWKLHahn} is
mediated by $\LLPO$ and relies freely on the results established in
$\BISH$ \cite{Bishop-Bridges-85,VarietiesConstMath}, the question
remains as to whether this equivalence holds over a weaker
intuitionistic base system with only restricted forms of countable
choice and hence compatible with $\RCA$. 

In this paper, we show that $\MIN$ and $\WKL$ are equivalent over the
base system $\ELZero$ of constructive reverse mathematics
\cite{ConstRevMatheCompactness,IshiharaEtalBinaryExpansion}.
As a corollary, we obtain the equivalence between $\MIN$ and the
\emph{Cantor intersection theorem} ($\CIT$) \cite{BergerIshiharaFanUniqExistence},
which was previously shown to be equivalent over $\BISH$
\cite{IshiharaLLPOWKLHahn}. Since $\WKL$ decomposes into $\LLPO$ and
$\PiOneDJDC$ \cite[Corollary 28]{BergerIshiharaSchusterWKLFanDeMorganDependentChoice},%
\footnote{
  $\PiOneDJDC$ denotes \emph{Dependent Choice for disjunctions of
  $\Pi^{0}_{1}$ formulas}:
  \begin{align*}
    \PiOneDJDC : 
    \forall u \left( A_0(u) \lor  A_1(u) \right)
    \imp
    \exists \alpha \forall n  A_{\alpha(n)}(\overline{\alpha}n),
  \end{align*}
  where $A_i(u)$ is a formula of the form $\forall n B(u,n)$.
}
our result identifies not only the logical principle (i.e., $\LLPO$)
but also the underlying choice axiom of $\MIN$ and $\CIT$.

Moreover, we show that the equivalence between $\WKL$, $\MIN$, and
$\CIT$ naturally extends to the equivalence between the corresponding
principles with unique existence assumptions, which are then shown to
be equivalent to the Fan theorem for detachable bars and the
positivity property for a real-valued uniformly continuous function on
a compact space. The results sharpen the equivalence among these
principles obtained by Berger and Ishihara
\cite{BergerIshiharaFanUniqExistence} in $\BISH$.

\section{Formal base system}
We adopt $\ELZero$ \cite[Section 2]{IshiharaEtalBinaryExpansion}
as our base system for constructive reverse mathematics.
$\ELZero$ is a subsystem of elementary analysis $\EL$ \cite[3.6]{ConstMathI},
which is based on two-sorted intuitionistic logic (one
sort for natural numbers and the other for functions on natural
numbers). $\ELZero$ is obtained from $\EL$ by restricting the induction
scheme to quantifier-free formulas. The classical system $\RCA$
is obtained from $\ELZero$ by adding the law of excluded middle.  For
a detailed description of $\ELZero$, the reader is referred to
\cite[Section 2]{IshiharaEtalBinaryExpansion}. 
For the purpose of our investigation, however, some comments are in
order.

First, $\ELZero$ has only the \emph{quantifier-free axiom of choice}:
\begin{align}
  \label{qfAC}
  \qfAC: \forall m \exists n A(m,n) \imp \exists \alpha \forall m
  A(m,\alpha(m)),
\end{align}
where $m,n$ range over natural numbers and $\alpha$ ranges over
functions on natural numbers, and $A$ is a quantifier-free formula that
does not contain $\alpha$ as a free variable.
Using pairing function, one can derive the choice principle
$\SigmaOneAC$ in which the formula $A(m,n)$ in \eqref{qfAC} is of the
form $\exists k  B(m,n,k)$ where $B$ is quantifier free.

Second, since $\ELZero$ has only two sorts, a uniformly continuous
function on real numbers (or more generally Euclidean spaces) must be
defined in terms of a function on natural numbers
(see \cite[Section 5]{Ishihara2009Relativize} for the details of the encoding).
However, as in \cite{IshiharaEtalBinaryExpansion}, we work with
uniformly continuous functions on Euclidean spaces as if such higher
type objects are directly available in $\ELZero$.

\paragraph{Notations}
We follow the notation in \cite{ConstMathI,ConstMathII}.
The letters $i,j,k,l,m,n,\dots$ range over the set $\Nat$ of natural numbers,
and $\alpha,\beta,\dots$ range over the set $\Nat^{\Nat}$ of functions
on natural numbers.
Using a pairing function, one can code finite sequences of natural
numbers in $\ELZero$. The set of finite sequences of natural numbers
is denoted by $\Seq$, the set of binary sequences is denoted by
$\Bin$, and letters $a,b,c,\dots$ range over $\Seq$.
Also, $\nil$
denotes the empty sequence, and $\langle i_0, \dots, i_{n-1} \rangle$
denotes a finite sequence of length $n$. 
The concatenation of finite
sequences $a$ and $b$ is denoted by $a*b$, the length of $a$ is
denoted by $|a|$, for $i < |a|$, the $i$-th component of $a$ is
denoted by $a_i$, and $a \preceq b$ means that $a$ is an initial
segment of $b$. For an infinite sequence $\alpha$,
$\overline{\alpha}n$ denotes the initial segment of $\alpha$ of
length $n$,
namely $\langle \alpha(0),\dots,\alpha(n-1)\rangle$.
We also define $\overline a n = \langle a_{0},\dots,a_{n-1} \rangle$
for a finite sequence $a$ whenever $n \leq \lth{a}$.
For a finite sequence $a$, $\widehat{a}$ denotes the infinite sequence
where $a$ is followed by an infinite sequence of $0$s.
For a finite sequence $a$ and an infinite sequence $\alpha$, $\alpha
\in a$ means that $a$ is an initial segment of $\alpha$.

\section{Minimum principle}
\label{sec:MIN}
First, we recall some notions related to $\WKL$.
A \emph{tree} is an inhabited and detachable subset%
\footnote{We identify a detachable subset with its characteristic
function.} $T \subseteq \Seq$ such that 
$a \in T \land b \preceq a \imp b \in T$
for all $a, b \in \Seq$. Let $T$ be a tree:
\begin{itemize}
  \item $T$ is \emph{binary} if $T \subseteq \Bin$.

  \item $T$ is \emph{infinite} if $\forall n \in \Nat \exists a
    \in \Seq \left( |a| = n \land a \in T \right)$.

  \item An \emph{infinite path} of $T$ is a sequence 
    $\alpha \in \Baire$ such that
    $\forall n \in \Nat \; \overline{\alpha}n \in T$.
\end{itemize}
\emph{Weak K\"{o}nig's lemma} ($\WKL$) is the following principle:
\begin{description}
  \item[$\WKL$:]
    Every infinite binary tree has an infinite path.
\end{description}
\begin{theorem}
  \label{thm:WKL_MIN}
  The following are equivalent in $\ELZero$.
  \begin{enumerate}
    \item\label{MIN} $\MIN$,
    \item\label{WKL} $\WKL$.
  \end{enumerate}
\end{theorem}
\begin{proof}
  $\MIN \Rightarrow \WKL$:
  Assume $\MIN$, and let $T$ be an infinite binary tree.
  For each $k \in \Nat$, let
  \begin{equation}
    \label{def:Tk}
    T_k =  \left\{ a \in T \mid |a| = k \right\},
  \end{equation}
  and let $\chi_k$ be the characteristic function
  of $\Bin \backslash T_k$.
  Define a uniformly continuous function $f \colon \Cantor \to \Real$ by 
  \[
    f(\alpha) = \sum_{i = 0}^{\infty}
    2^{-i} \chi_i(\overline{\alpha}i).
  \]
  Clearly, $f(\alpha) = 0$ if and only if $\alpha$ is an infinite path
  of $T$. Since $T$ is infinite, we have $\inf f = 0$; hence,
  by $\MIN$, there exists $\alpha$ such that
  $f(\alpha) = 0$. Then $\alpha$ is an infinite path of $T$.
  \medskip

  $\WKL \Rightarrow \MIN$:
  Assume $\WKL$. Since every compact metric space can be represented as a
  uniformly continuous image of the Cantor space
  $\Cantor$~\cite[7.4.4]{ConstMathII},%
  \footnote{
    Specifically, each compact metric space is a uniformly continuous
    image of a finitely branching spread with an explicit height-wise
    bounding function \cite[7.4.3]{ConstMathII}, which in turn is a
    continuous image of $\Cantor$ \cite[4.7.5]{ConstMathI}
    (see also \cite{FujiwaraDFT}).
    }
  it is sufficient to prove the
  version of $\MIN$ for real-valued functions on the Cantor space
  whose infimum is $0$. So let $f \colon \Cantor \to \Real$ be a
  uniformly continuous function with modulus $\omega \colon \Nat \to
  \Nat$ of uniform continuity. Without loss of generality, we assume
  that $n \leq \omega(n) \leq \omega(n+1)$ for each $n \in \Nat$ and
  that $\inf f = 0$.

  For each $a \in \Bin$, define
  \begin{align*}
    T_a &\defeql \left\{ a*b \mid b \in \Bin  \right\}, 
    \\
    \widehat{T_a} &\defeql \left\{ \alpha \in \Cantor \mid \alpha \in a
    \right\}.
  \end{align*}
  Note that $\widehat{T_a}$ is a compact subspace of $\Cantor$, so
  the restriction $f|_{\widehat{T_a}}$ of $f$ to $\widehat{T_a}$ has an
  infimum. 
  By $\SigmaOneAC$, we have a function $\lambda \colon \Bin
  \to \Two$ such that
  \begin{equation}
    \label{def:lambda}
    \begin{aligned}
      \lambda(a) = 0 &\imp \inf f|_{\widehat{T_a}} < 2^{-|a|}, \\
      \lambda(a) = 1 &\imp \inf f|_{\widehat{T_a}} > 2^{-(|a| + 1)}.
    \end{aligned}
  \end{equation}
  Define $T \subseteq \Bin$ by
  \begin{align}
    \label{T:WKL-MIN}
    T \defeql \left\{ a \in \Bin \mid \lambda(a) = 0 \right\}.
  \end{align}
  We show that $T$ is an infinite binary tree.
  First, since $\inf f|_{\widehat{T_{\nil}}} = \inf f = 0 < 1 = 2^{0}$,
  we have $\nil \in T$.
  Second, let $a \in \Bin$ and $i \in \Two$, and suppose 
  $a * \langle i \rangle \in T$. Then,
  $\inf f|_{\widehat{T_{a * \langle i \rangle}}} <  2^{-(|a| + 1)}$.
  Since $T_a = T_{a*\langle 0 \rangle} \cup T_{a*\langle 1 \rangle}$,
  we have
  \[
    \inf f|_{\widehat{T_{a}}}
    =
    \min\left\{ 
      \inf f|_{\widehat{T_{a * \langle 0 \rangle}}},
      \inf f|_{\widehat{T_{a * \langle 1 \rangle}}}
    \right\}
    \leq
    \inf f|_{\widehat{T_{a * \langle i \rangle}}}
    < 2^{-(|a| + 1)}.
  \]
  Thus $\lambda(a) = 1$ is impossible. Hence $\lambda(a) = 0$,
  and so $a \in T$. Thus $T$ is a binary tree.
  To see that $T$ is infinite, fix $n \in \Nat$, and 
  suppose that $\lambda(a) = 1$ for all $a \in \Bin$ such that $|a| =
  n$. Since
  $\Bin = \bigcup\left\{ T_a \mid |a| = n \right\}$,
  we have 
  \[
    0 = \inf f = \min \left\{ \inf f|_{\widehat{T_{a}}} \mid |a| = n \right\}
    >  2^{-(n+1)},
  \]
  a contradiction. Thus there exists $a \in \Bin$ such that $|a|=n$
  and $\lambda(a) = 0$, i.e., $a \in T$. Hence $T$ is infinite.

  By $\WKL$, $T$ has an infinite path $\alpha \in \Cantor$. Fix $n \in
  \Nat$. Since $\lambda(\overline{\alpha}\omega(n)) = 0$, we have
  \[
    \inf f|_{\widehat{T_{\overline{\alpha}\omega(n)}}} < 2^{-\omega(n)}
    \leq 2^{-n}.
  \]
  Thus there exists $\beta \in \overline{\alpha}\omega(n)$ such 
  that $f(\beta) < 2^{-n}$.
  Since $\omega$ is a modulus of $f$, 
  we have $|f(\overline{\alpha}\omega(n)) - f(\beta)| < 2^{-n}$.
  Thus 
  \[
    f(\overline{\alpha}\omega(n)) < 2^{-n} + f(\beta) < 2^{-n} +
    2^{-n} = 2^{-(n-1)}.
  \]
  Hence, the sequence $\langle
  f(\overline{\alpha}\omega(n))
  \rangle_{n \in \Nat}$ converges to $0$. Since
  $\langle \overline{\alpha}\omega(n) \rangle_{n \in \Nat}$ converges
  to $\alpha$, we have $f(\alpha) = 0$ by the continuity of $f$.
\end{proof}

\section{Cantor intersection theorem}
\label{sec:CIT}
Working in $\BISH$, Ishihara \cite{IshiharaLLPOWKLHahn} showed that
the minimum principle is equivalent to the \emph{Cantor intersection
theorem}($\CIT$) \cite{BergerIshiharaFanUniqExistence}:%
\footnote{
  In \cite{IshiharaLLPOWKLHahn},
  the Cantor intersection theorem is called 
  the \emph{pseudo Heine--Borel principle}.
  Classically, this statement is the contrapositive of the Heine--Borel
  theorem.
}
\begin{description}
  \item[$\CIT$:]
    Every decreasing sequence $\langle F_i \rangle_{i \in \Nat}$ of
    inhabited, closed, and located subsets of a compact metric space
    has a point in common, i.e., $\bigcap_{i \in \Nat} F_i$ is
    inhabited.
\end{description}
As a corollary of Theorem \ref{thm:WKL_MIN},
we obtain the equivalence between $\MIN$ and $\CIT$
over $\ELZero$.
\begin{theorem}
  \label{thm:MIN_CIT}
  The following are equivalent over $\ELZero$:
  \begin{enumerate}
    \item $\MIN$,
    \item $\CIT$.
  \end{enumerate}
\end{theorem}
\begin{proof}
  $\MIN \Rightarrow \CIT$:
  The proof in \cite{IshiharaLLPOWKLHahn} carries over to $\ELZero$.
  We repeat the proof for the sake of completeness.
  Assume $\MIN$, and let $\langle F_i \rangle_{i \in \Nat}$ be a
  decreasing sequence of inhabited, closed, and  located subsets of a
  compact metric space $X$. Define a function $f \colon X \to
  \Real$ by 
  \begin{align}
    \label{f:MIN-CIT}
    f(x) = \sum_{i = 0}^{\infty} 2^{-i} d(x,F_i),
  \end{align}
  which is clearly uniformly continuous. Since
  $\langle F_i \rangle_{i \in \Nat}$ is decreasing and each $F_i$ is
  inhabited, we have $\inf f = 0$. Hence, by $\MIN$, there exists $x
  \in X$ such that $f(x) = 0$, so that $d(x,F_{i}) = 0$ for all $i \in
  \Nat$.  Since $F_i$ is closed, we have $x \in \bigcap_{i \in \Nat}
  F_i$.
  \medskip

  $\CIT \Rightarrow \MIN$:
  By Theorem \ref{thm:WKL_MIN}, it suffices to show that $\CIT$ implies
  $\WKL$. Let $T$ be an infinite binary tree. For each $i \in \Nat$,
  define
  \begin{align}
    \label{F:CIT-MIN}
    F_i = \left\{ \alpha \in \Cantor 
    \mid \overline{\alpha}i \in T \right\}.
  \end{align}
  Then, $F_i$ is a closed and located subset of $\Cantor$, which
  is inhabited since $T$ is infinite.
  Since $T$ is a tree, the sequence $\langle F_{i} \rangle_{i \in
  \Nat}$ is decreasing.
  Thus, by $\CIT$, there exists $\alpha$ such that $\alpha \in \bigcap_{i
  \in \Nat} F_i$, i.e., $\overline{\alpha}n \in T$ for all $n \in
  \Nat$. Hence, $T$ has an infinite path.
\end{proof}

As a corollary, we obtain the contrapositive
of the equivalence between the Fan theorem
and the Heine--Borel covering theorem \cite[Theorem 5.4]{Loeb05}.
\begin{corollary}
  \label{cor:CIT_WKL}
  The following are equivalent over $\ELZero$:
  \begin{enumerate}
    \item $\CIT$,
    \item $\WKL$.
  \end{enumerate}
\end{corollary}
\begin{proof}
  By Theorem \ref{thm:WKL_MIN}
  and Theorem \ref{thm:MIN_CIT}.
\end{proof} 

\section{Unique existence}
\label{sec:UniquExistence}
Working in $\BISH$, Berger and Ishihara
\cite{BergerIshiharaFanUniqExistence} showed that the variants of
$\WKL$, $\MIN$, and $\CIT$ with a unique existence assumption, called
$\WKL!$, $\MIN!$, and $\CIT!$ respectively, are mutually equivalent.
Moreover, they showed that these principles are all equivalent to the
\emph{Fan theorem} ($\FAN$) (i.e., a contrapositive of $\WKL$
\cite{IshiharaWKLImpDFan}) and the \emph{positivity property} ($\POS$)
of a real-valued uniformly continuous function on a compact space.

We derive the equivalence between $\WKL!$, $\MIN!$, and $\CIT!$ in
$\ELZero$ as a corollary of the results in Section~\ref{sec:MIN} and
Section~\ref{sec:CIT}.
For the sake of completeness, we also show that
$\WKL!$, $\MIN!$, and $\CIT!$ are all equivalent to 
$\FAN$ and $\POS$.

We first recall the precise formulations of the relevant principles.
\begin{enumerate}
  \item A unique version of the Cantor intersection theorem:
    \begin{description}
      \item[$\CITB$:]
        Every decreasing sequence $\langle F_i \rangle_{i \in \Nat}$
        of inhabited, closed, and located subsets of a compact metric
        space having at most one point in common has a point in common,
        i.e., $\bigcap_{i \in \Nat} F_i$ is inhabited.

    \end{description}
    Here, a decreasing sequence of subsets \emph{has at most one point in common}
    if
    \[
      x \neq y \imp \exists i \left( d(x,F_i) > 0 \lor d(y,F_i) > 0 \right).
    \]

  \item A unique version of the minimum principle:
    \begin{description}
      \item[$\MINB$:]
        Every real-valued uniformly continuous function
        $f \colon X \to \Real$ on a compact metric space $X$ having 
        at most one minimum point attains its minimum: i.e., there
        exists $z \in X$ such that $f(z) = \inf f$.
    \end{description}
    Here, a function \emph{has at most one minimum point}
    if
    \[
      x \neq y \imp  \left( f(x) > \inf f \lor f(y) > \inf f \right).
    \]

  \item A unique version of weak K\"onig lemma:
    \begin{description}
      \item[$\WKLB$:]
        Every infinite binary tree having at most one infinite path has an infinite path.
    \end{description}
    Here, a tree \emph{has at most one infinite path} if
    \[
      \alpha \neq \beta \imp \exists n \left( \overline{\alpha}n \notin T \lor \overline{\beta}n \notin T  \right).
    \]

  \item The Fan theorem:
    \begin{description}
      \item[$\FAN$:] Each detachable bar is uniform.
    \end{description}
    Here, a detachable subset%
    \footnote{A detachable subset is a subset defined by its characteristic
    function.} 
    $B \subseteq \Bin$ is a \emph{bar} if
    \[
      \forall \alpha \in \Cantor \exists n \left( \overline{\alpha}n
      \in B \right).
    \]
    A bar is \emph{uniform} if 
    \[
      \exists N \forall \alpha \in \Cantor \exists n \leq N \left( \overline{\alpha}n
      \in B \right).
    \]

  \item The positivity property:
    \begin{description}
      \item[$\POS$:] Each positive valued uniformly continuous
        function $f \colon X \to \Real$ on a compact metric space $X$
        has a positive infimum.
    \end{description}
\end{enumerate}

\begin{theorem}
  \label{thm:UniqueExistence}
  The principles $\WKL!$, $\MIN!$, $\CIT!$, $\FAN$, and $\POS$
  are all equivalent over $\ELZero$.
\end{theorem}
\begin{proof}
   $\WKL! \Rightarrow \MIN!$:
   The proof is identical to that of the direction
   ($\WKL \to \MIN$) of Theorem \ref{thm:WKL_MIN},
   except that one needs to show that the tree
   $T$ defined by \eqref{T:WKL-MIN} has at most one infinite path if
   the function $f \colon \Cantor \to \Real$ has at most one minimum point.

   To see this, suppose that $f$ has at most one minimum point.
   Let $\alpha,\beta \in \Cantor$ be such that $\alpha \neq \beta$.
   Then either $f(\alpha) > 0$ or $f(\beta) > 0$.

   If $f(\alpha) > 0$, then there exists $n \in \Nat$
   such that $f(\alpha) > 2^{-n}$. Then, for any $\beta \in
   \overline{\alpha}\omega(n+1)$, we have
   \[
     f(\beta) > f(\alpha) - 2^{-(n+1)} > 2^{-(n+1)} \geq
     2^{-\omega(n+1)}.
   \]
   Thus $\inf f|_{\widehat{T_{\overline{\alpha}\omega(n+1)}}} \geq
   2^{-\omega(n+1)}$, and so $\lambda(\overline{\alpha}\omega(n+1)) =
   0$ is impossible. Hence $\lambda(\overline{\alpha}\omega(n+1)) =
   1$, i.e., $\overline{\alpha}\omega(n+1) \notin T$.
   Similarly, if $f(\beta) > 0$, then there exists $n$
   such that $\overline{\beta}n \notin T$.
   \medskip

   $\MIN! \Rightarrow \CIT!$:
   Again, the proof is identical to that of the direction
   ($\MIN \to \CIT$) of Theorem \ref{thm:MIN_CIT},
   except that one needs to show that the function
   $f$ defined by \eqref{f:MIN-CIT} has at most one minimum point
   if the sequence $\langle F_i \rangle_{i \in \Nat}$
   has at most one point in common.
   But this is immediately seen from the fact that $f(x) > 0$ if and
   only if $\exists i \in \Nat \,d(x,F_i) > 0$.
   \medskip

   $\CIT! \Rightarrow \WKL!$:
   Again, the proof is identical to that of the direction
   ($\CIT \to \WKL$) of Theorem \ref{thm:MIN_CIT},
   except that one needs to show that
   the sequence $\langle F_i \rangle_{i \in \Nat}$
   defined by \eqref{F:CIT-MIN} has at most one point in common
   if the tree $T$ has at most one infinite path.
   But this is immediately seen from the fact that
   $d(\alpha,F_i) > 0$ if and only if $\overline{\alpha}i \notin T$.
   \medskip

   $\WKL! \Rightarrow \FAN$:
   The proof in \cite{BergerIshiharaFanUniqExistence} carries over to
   $\ELZero$. We recall the proof with a slight modification for the
   sake of completeness.

   Let $B \subseteq \Bin$ be a detachable bar. Without loss of
   generality, we may assume that $B$ is monotone, i.e., 
   $a \in B \land a \preceq b \imp b \in B$. Define a detachable
   subset $T \subseteq \Bin$ by 
   \[
     a \in T \defeqiv
       \forall m \leq |a|
       \left( \overline{a}m \in B \imp
         \forall b \in \Bin
         \left( |b| = m \imp b \in B \right)
         \land 
         a = \langle 0, \ldots, 0 \rangle
       \right),
   \]
   where $\langle 0, \ldots, 0 \rangle$ is a sequence of zeros.
   It is straightforward to see that $T$ is an infinite tree.
   We show that $T$ has at most one infinite path. Let $\alpha,\beta
   \in \Cantor$ be such that $\alpha \neq \beta$.
   Since $B$ is monotone bar, there exists 
   $n$ such that $\overline{\alpha}n, \overline{\beta}n \in B$
   and $\overline{\alpha}n \neq \overline{\beta}n$.
   If $\overline{\alpha}n, \overline{\beta}n \in T$, then
   $\overline{\alpha}n = \langle 0, \dots, 0 \rangle = \overline{\beta}
   n$, a contradiction. Thus either 
  $\overline{\alpha}n \notin T$ or $\overline{\beta}n \notin T$
  as required. By $\WKLB$, $T$ has an infinite path $\alpha$.
  Since $B$ is a bar, there exists $N$ such that $\overline{\alpha}N
  \in B$. Then, $N$ witnesses the uniformity of $B$.
   \medskip

   $\FAN \Rightarrow \POS$:
   As in the proof of $\WKL \Rightarrow \MIN$ in 
   Theorem~\ref{thm:WKL_MIN}, it suffices to prove a special case of $\POS$ 
   where the domain of a function is $\Cantor$.
   Assume $\FAN$, and let $f \colon \Cantor \to \Real$ be positive
   valued uniformly continuous function. 
   Let $\lambda$ be the function defined by \eqref{def:lambda}, and
   let
   \[
     B = \left\{ a \in \Bin \mid \lambda(a) = 1 \right\}.
   \]
   By an argument similar to that used in the proof of $\MINB \Rightarrow \CITB$,
   one can see that $B$ is a bar. By $\FAN$, there exists $N$ such
   that $\forall \alpha \in \Cantor \exists n \leq N \overline{\alpha}n \in B$.
   Then  $\inf f \geq 2^{-(N+1)}$.
   \medskip

   $\POS \Rightarrow \WKLB$:
   We essentially follow the proof of $\POS \Rightarrow \CITB$ in
   \cite{BergerIshiharaFanUniqExistence}.
   Assume $\POS$, and let $T$ be an infinite binary tree with at most
   one infinite path. 
   We show that
   \begin{equation}
     \label{C:POT-WKLB}
     \forall n \exists k \geq n
     \forall a,b \in T_k \left( \overline{a}n = \overline{b}n  \right),
   \end{equation}
   where $T_k$ is given as in \eqref{def:Tk}.
   Fix $n \in \Nat$, and let
   \begin{align*}
     Y_n &= \left\{ (a,b) \in \Bin \times \Bin \mid 
             |a| = |b| = n \land a \neq b \right\}, \\
     \widehat{Y_n} &= \left\{ (\alpha,\beta) \in \Cantor \times
     \Cantor \mid \exists (a,b) \in Y_n \left( \alpha \in a \land
   \beta \in b \right)\right\}.
   \end{align*}
   Then $\widehat{Y_n}$ is a closed located subset of $\Cantor \times
   \Cantor$. 
   Define a uniformly continuous $f \colon \widehat{Y_n} \to \Real$ by 
   \[
     f(\alpha,\beta) =
     \sum_{i = 0}^{\infty} 2^{-(i+1)} \max\left\{
       \chi_i(\overline{\alpha}i),
     \chi_i(\overline{\beta}i) \right\},
   \]
   where $\chi_i$ be the characteristic function of $\Bin \backslash
   T_i$.
   Since $T$ has at most one infinite path, $f$ is positive valued by
   the definition of $\widehat{Y_n}$.  Thus, there exists $k \geq n$
   such that $\inf f > 2^{-k}$.  We show that $k$ witnesses 
   \eqref{C:POT-WKLB}. To this end, let $a,b \in T_k$, and suppose that $\overline{a}n
   \neq \overline{b}n$. Then $(\widehat{a},\widehat{b}) \in
   \widehat{Y_n} $ and so $f(\widehat{a},\widehat{b}) > 2^{-k}$. But
   $a,b \in T_k$ implies $f(\widehat{a},\widehat{b}) \leq 2^{-k}$, a
   contradiction.  Hence $\overline{a}n = \overline{b}n$ as required.

   Since $T$ is infinite, by $\qfAC$, there exists 
   a sequence $\langle a_n \rangle_{n \in \Nat}$ such that $a_n \in T_n$
   for all $n \in \Nat$. Then, by \eqref{C:POT-WKLB}, 
   $\langle \widehat{a_n} \rangle_{n \in \Nat}$ is a Cauchy sequence
   in $\Cantor$, and hence, it converges to $\alpha$. Since 
   $a_n \in T_n$ for each $n \in \Nat$, we have $\forall n
   \left( \overline{\alpha}n \in T \right)$, i.e., $\alpha$ is an infinite path in
   $T$.
\end{proof}

\subsection*{Acknowledgements}
The work is supported by JSPS KAKENHI Grant Number JP23K03197
and JP21KK0045.
\bibliographystyle{abbrv}
\bibliography{$HOME/refs.bib}
\end{document}